\documentclass[12pt,letterpaper,showkeys]{article} 
\usepackage{graphicx}
\usepackage{amssymb}
\usepackage{amsmath}
\usepackage{mathrsfs}
\graphicspath{{figures/}}
\newtheorem{theorem}{Theorem}[section] 
\newtheorem{lemma}[theorem]{Lemma}

\newtheorem{remark}[theorem]{Remark}
\newtheorem{assumption}{Assumption} 
\newtheorem{definition}{Definition}
\newcommand{\beq}{\begin{equation}} 
\newcommand{\eeq}{\end{equation}} 
\newcommand{\beqa}{\begin{eqnarray}} 
\newcommand{\eeqa}{\end{eqnarray}} 
\newcommand{\beqas}{\begin{eqnarray*}} 
\newcommand{\eeqas}{\end{eqnarray*}} 
\newcommand{\ba}{\begin{array}} 
\newcommand{\ea}{\end{array}} 
\newcommand{\bi}{\begin{itemize}} 
\newcommand{\ei}{\end{itemize}} 
\newcommand{\gap}{\hspace*{2em}} 
 
\newcommand{\nn}{\nonumber}

\def\eqnok#1{(\ref{#1})}

\def\vgap{\vspace*{.1in}}
\def\QED{\ifhmode\unskip\nobreak\fi\ifmmode\ifinner\else\hskip5pt\fi\fi
  \hbox{\hskip5pt\vrule width5pt height5pt depth1.5pt\hskip1pt}} 
\def\Arg{{\rm Arg}}

\def\bx{{\bar x}}

\def\cB{{\cal B}}

\def\cN{{\cal N}}

\def\cK{{\cal K}}
 
\def\cU{{\cal U}}

\def\eps{{\epsilon}}

\def\tx{{\tilde x}}

\title{Iterative Hard Thresholding Methods for  $l_0$ Regularized Convex Cone Programming}  

\author{
	Zhaosong Lu%
	\thanks{
	Department of Mathematics, Simon Fraser University, Burnaby, BC, 
	V5A 1S6, Canada. (email: {\tt zhaosong@sfu.ca}). This work was supported 
        in part by NSERC Discovery Grant. Part of this work was conduct during the author's 
        sabbatical leave in Department of Industrial and Systems Engineering at 
       Texas A\&M  University. The author would like to thank them for hosting his visit.} 
	}

\date{October 30, 2012}

\begin{document}

\maketitle

\begin{abstract}

In this paper we consider $l_0$ regularized  convex cone programming problems. In particular, we first 
propose an  iterative hard thresholding (IHT) method and its variant for solving $l_0$ regularized box 
constrained convex programming. We show that the sequence generated by these methods converges 
to a local minimizer. Also, we establish the iteration complexity  of the IHT method for finding an 
$\eps$-local-optimal solution. We then propose a method for solving $l_0$ regularized  convex cone 
programming by applying the IHT method to its quadratic penalty relaxation and establish its iteration 
complexity for finding an $\eps$-approximate local minimizer. Finally, we propose a variant of this 
method in which the associated penalty parameter is dynamically updated, and show that every 
accumulation point is a local minimizer of the problem.

\vskip14pt


\vskip14pt
\end{abstract}

\section{Introduction} \label{intro}

Sparse approximations have over the last decade gained a great deal of popularity in 
numerous areas.  For example, in compressed sensing, a large sparse signal is decoded 
by finding a sparse solution to a system of linear equalities and/or inequalities. Our particular 
interest of this paper is to find a sparse approximation to a convex cone programming problem 
in the form of
\beq \label{cone}
\ba{ll}
\min & f(x)  \\
\mbox{s.t.} & A x - b \in \cK^*, \\
& l \le x \le u  
\ea 
\eeq
for some $l \in \bar\Re^n_{-}$, $u \in \bar\Re^n_{+}$, $A \in \Re^{m \times n}$ and 
$b \in \Re^m$, where $\cK^*$ denotes the dual cone of a closed convex cone $\cK \subseteq \Re^m$, 
i.e., $\cK^* = \{s\in \Re^m: s^Tx \ge 0, \forall x \in \cK\}$, and $\bar\Re^n_-=\{x:-\infty \le x_i \le 0, 
\  1 \le i \le n\}$ and $\bar\Re^n_+=\{x\in \Re^n: 0 \le x_i \le \infty, \ 1 \le i \le n\}$.  A sparse 
solution to \eqref{cone} can be sought by solving the following $l_0$ regularized convex cone 
programming problem:
\beq \label{l0-cone}
\ba{ll}
\min & f(x) + \lambda \|x\|_0 \\
\mbox{s.t.} & A x - b \in \cK^*, \\
& l \le x \le u  
\ea 
\eeq
for some $\lambda >0$, where $\|x\|_0$ denotes the cardinality of $x$. One special case of 
\eqref{l0-cone}, that is, the $l_0$-regularized unconstrained least squares problem, has been 
well studied in the literature (e.g., \cite{Nik11,LuZh12}), and some methods were developed for 
solving it. For example, the iterative hard thresholding (IHT) methods \cite{HeGiTr06,BlDa08,BlDa09} 
and matching pursuit algorithms \cite{MaZh93,Tr04} were proposed to solve this type of 
problems. Recently, Lu and Zhang \cite{LuZh12} proposed a penalty decomposition method for solving 
a more general class of $l_0$ minimization problems.

As shown by the extensive experiments in \cite{BlDa08,BlDa09}, the IHT method performs very well in 
finding a sparse solution to unconstrained least squares problems. In addition, the similar type of methods \cite{CaCaSh10,JaMeDh10} were successfully applied to find low rank solutions in the 
context of matrix completion.  Inspired by these works, in this paper we study IHT methods 
for solving $l_0$ regularized  convex cone programming problem \eqref{l0-cone}. In particular, 
we first propose an IHT method and its variant for solving $l_0$ regularized box constrained 
convex programming. We show that the sequence generated by these methods converges to 
a local minimizer. Also, we establish the iteration complexity  of the IHT method for finding an 
$\eps$-local-optimal solution. We then propose a method for solving $l_0$ regularized  convex 
cone programming by applying the IHT method to its quadratic penalty relaxation and establish 
its iteration complexity for finding an $\eps$-approximate local minimizer of the problem. We 
also propose a variant of the method in which the associated penalty parameter is dynamically 
updated, and show that every accumulation point is a local minimizer of the problem.

The outline of this paper is as follows. In  Subsection \ref{notation} we introduce some notations 
that are used in the paper. In Section \ref{tech} we present some technical results about a projected 
gradient method for convex programming.  In Section \ref{l0-box} we propose IHT methods for solving 
$l_0$ regularized box constrained convex programming and study their convergence.  In section 
\ref{l0-cp} we develop IHT methods for solving $l_0$ regularized  convex cone programming and 
study their convergence. Finally, in Section \ref{conclude} we present some concluding remarks. 

 \subsection{Notation} \label{notation}

Given a nonempty closed convex $\Omega \subseteq \Re^n$ and  an arbitrary point 
$x \in \Omega$, $\cN_\Omega(x)$ denotes the normal cone of $\Omega$ at $x$. In addition, 
$d_\Omega(y)$ denotes the Euclidean distance between $y \in \Re^n$ and $\Omega$. All norms used 
in the paper are Euclidean norm denoted by $\|\cdot\|$. We  use $\cU(r)$ to denote a ball 
centered at the origin with a radius $r \ge 0$, that is, $\cU(r): = \{x\in\Re^n: \|x\| \le r\}$.

\section{Technical preliminaries} \label{tech}

In this section we present some technical results about a projected gradient method for convex 
programming that will be subsequently used in this paper.

Consider the convex programming problem 
\beq \label{cp-phi}
\phi^* := \min\limits_{x\in X} \phi(x), 
\eeq 
where $X \subseteq \Re^n$ is a closed convex set and $\phi: X \to \Re$ is 
a smooth convex function whose gradient is Lipschitz continuous with constant $L_\phi>0$. 
Assume that the set of optimal solutions of \eqref{cp-phi}, denoted by $X^*$, is nonempty.

Let $L \ge L_\phi$ be arbitrarily given. A projected gradient of $\phi$ at any $x\in X$ with respect 
to $X$ is defined as 
\beq \label{gx}
g(x) := L\left[x - \Pi_X\left(x-\nabla \phi(x)/L\right)\right],
\eeq 
where $\Pi_X(\cdot)$ is the projection map onto $X$ (see, for example, \cite{Nes04}).

The following properties of the projected gradient are essentially shown in Proposition 3 and 
Lemma 4 of \cite{LaMo12} (see also \cite{Nes04}).

\begin{lemma} \label{proj-grad}
Let $x\in X$ be given and define $x^+: = \Pi_X(x-\nabla \phi(x)/L)$. Then, 
for any given $\epsilon \ge 0$, the following statements hold:
\bi
\item[a)] $\|g(x)\| \le \epsilon$ if and only if $\nabla \phi(x) \in -\cN_X(x^+) + \cU(\epsilon)$. 
\item[b)] $\|g(x)\| \le \epsilon$ implies that $\nabla \phi(x^+) \in -\cN_X(x^+) + \cU(2\epsilon)$.
\item[c)] $\phi(x^+) -\phi(x) \le -\|g(x)\|^2/(2L)$.
\item[d)] $\phi(x) - \phi(x^*) \ge \|g(x)\|^2/(2L)$, where $x^* \in \Arg\min\{\phi(y): y\in X\}$.
\ei
\end{lemma} 

\vgap

We next study a projected gradient method for solving \eqref{cp-phi}.  

\gap

\noindent
{\bf Projected gradient method for \eqnok{cp-phi}:}  \\ [5pt]
Choose an arbitrary $x^0\in X$. Set $k=0$. 
\begin{itemize}
\item[1)] Solve the subproblem 
\beq \label{pg-subprob}
x^{k+1} = \arg\min\limits_{x\in X}\{\phi(x^k)+\nabla \phi(x^k)^T(x-x^k)+\frac{L}{2}\|x-x^k\|^2\}.
\eeq
\item[2)]
Set $k \leftarrow k+1$ and go to step 1). 
\end{itemize}
\noindent
{\bf end}

\vgap

Some properties of the above projected gradient method are established in the following 
two theorems, which will be used in the subsequent sections of this paper.

\begin{theorem} \label{sd-lemma} 
Let $\{x^k\}$ be generated by the above projected gradient method. Then the 
following statements hold:
\bi
\item[\rm (i)] 
For every $k \ge 0$ and $l \ge 1$, 
\beq \label{suff-reduct}
\phi(x^{k+l}) - \phi^* \le \frac{L}{2l} \|x^k-x^*\|^2.
\eeq
\item[\rm (ii)]
$\{x^k\}$ converges to some optimal solution $x^*$ of \eqref{cp-phi}. 
\ei
\end{theorem}

\begin{proof}
(i) Since the objective function of \eqref{pg-subprob} is strongly convex with modulus 
$L$, it follows that for every $x\in X$,
\[
\phi(x^k)+\nabla \phi(x^k)^T(x-x^k)+\frac{L}{2}\|x-x^k\|^2 \ge \phi(x^k)+
\nabla \phi(x^k)^T(x^{k+1}-x^k)+\frac{L}{2}\|x^{k+1}-x^k\|^2 + \frac{L}{2}\|x-x^{k+1}\|^2. 
\]
By the convexity of $\phi$, Lipschitz continuity of $\nabla \phi$ and $L \ge L_\phi$,  we have
\[
\ba{lcl}
\phi(x) &\ge & \phi(x^k)+\nabla \phi(x^k)^T(x-x^k), \\ [5pt]
\phi(x^{k+1}) &\le &\phi(x^k)+ 
\nabla \phi(x^k)^T(x^{k+1}-x^k)+\frac{L}{2}\|x^{k+1}-x^k\|^2, 
\ea
\]
which together with the above inequality imply that 
\beq \label{1st-opt}
\phi(x) +\frac{L}{2}\|x-x^k\|^2 \ \ge \ \phi(x^{k+1}) + \frac{L}{2}\|x-x^{k+1}\|^2, \ \ \ \forall x\in X. 
\eeq
Letting $x = x^k$ in \eqref{1st-opt}, we obtain that 
\[
\phi(x^k) - \phi(x^{k+1}) \ \ge \ L \|x^{k+1}-x^k\|^2/2.
\]
Hence, $\{\phi(x^k)\}$ is decreasing. 
Letting $x=x^* \in X^*$ in \eqref{1st-opt}, we have 
\[
\phi(x^{k+1}) - \phi^* \ \le \ \frac{L}{2} \left(\|x^k-x^*\|^2-\|x^{k+1}-x^*\|^2\right), 
\ \ \ \forall k \ge 0.
\]
Using this inequality and the monotonicity of $\{\phi(x^k)\}$, we obtain that 
\beq \label{sum-ineq}
l (\phi(x^{k+l})-\phi^*) \ \le \ \sum^{k+l-1}\limits_{i=k}[\phi(x^{i+1})-\phi^*]  \ \le \ 
\frac{L}{2} \left(\|x^k-x^*\|^2-\|x^{k+l}-x^*\|^2\right),
\eeq
which immediately yields \eqref{suff-reduct}.

(ii) It follows from \eqref{sum-ineq} that 
\beq \label{dist}
\|x^{k+l} -x^*\| \ \le \ \|x^k-x^*\|,  \ \ \ \forall k \ge 0, l \ge 1. 
\eeq
Hence, $\|x^k-x^*\| \le \|x^0-x^*\|$ for every $k$. It implies that $\{x^k\}$ is bounded. Then, there 
exists a subsequence $K$ such that $\{x^k\}_K \to \hat x^* \in X$. It can be seen from 
\eqref{suff-reduct} that $\{\phi(x^k)\}_K \to \phi^*$. Hence, $\phi(\hat x^*) = \lim_{k\in K \to \infty} 
\phi(x^k) = \phi^*$, which implies that $\hat x^* \in X^*$. Since \eqref{dist} holds for any $x^*\in 
X^*$, we also have $\|x_{k+l} -\hat x^*\| \ \le \ \|x^k-\hat x^*\|$ for every $k \ge 0$ and $l \ge 1$. 
This together with the fact $\{x^k\}_K \to \hat x^*$ implies that $\{x^k\} \to \hat x^*$ and hence 
statement (ii) holds. 
\end{proof}

\gap

\begin{theorem} \label{sd-thm}
Suppose that $\phi$ is strongly convex with modulus $\sigma >0$. Let $\{x^k\}$ be generated by 
the above projected gradient method. Then, for any given $\epsilon >0$, the following statements hold:
\bi
\item[\rm (i)]
 $\phi(x^k)-\phi^* \le  \epsilon$ whenever 
\[
k \ge  2\lceil  L/\sigma\rceil \left\lceil \log\frac{\phi(x^0)-\phi^*}{\epsilon}
\right\rceil.
\]
\item[\rm (ii)]  $\phi(x^k)-\phi^* <  \epsilon$ whenever 
\[
k \ge 2 \lceil  L/\sigma\rceil \left\lceil \log\frac{\phi(x^0)-\phi^*}{\epsilon} 
\right\rceil + 1.
\]
\ei
\end{theorem}
 
 \begin{proof}
(i) Let $M= \lceil  L/\sigma\rceil$. It follows from Theorem \ref{sd-lemma} that 
 \[
 \phi(x^{k+l}) - \phi^* \le \frac{L}{2l} \|x^k-x^*\|^2 \le \frac{L}{\sigma l} (\phi(x^k)-\phi^*),
 \]
where $x^*$ is the optimal solution of \eqref{cp-phi}. Hence, we have
\[
\phi(x^{k+2M}) - \phi^* \le \frac{L}{2\sigma M} (\phi(x^k)-\phi^*) \ \le \ \frac12 (\phi(x^k)-\phi^*),
\]
which implies that 
\[
\phi(x^{2jM}) - \phi^* \le \frac{1}{2^j} (\phi(x^0)-\phi^*).
\] 
Let $K=\lceil \log((\phi(x^0)-\phi^*)/\epsilon)\rceil$. Hence, when $k \ge 2KM$, we have
\[
\phi(x^k) - \phi^* \ \le \ \phi(x^{2KM})-\phi^* \ \le \ \frac{1}{2^K}(\phi(x^0)-\phi^*) \ \le \ \epsilon,
\] 
 which immediately implies that statement (i) holds.

(ii) Let $K$ and $M$ be defined as above. If $\phi(x^{2KM})=\phi^*$, by monotonicity of $\{\phi(x^k)\}$ 
we have $\phi(x^k)=\phi^*$ when $k > 2KM$, and hence the conclusion holds. We 
now suppose that $\phi(x^{2KM}) >\phi^*$. It implies that $g(x^{2KM}) \neq 0$, where $g$ is defined 
in \eqref{gx}. Using this relation,  Lemma \ref{proj-grad} (c) and statement (i), we obtain that 
$\phi(x^{2KM+1}) < \phi(x^{2KM}) \le \eps$, which together with the montonicity of $\{\phi(x^k)\}$ 
implies that the conclusion holds.  
 \end{proof}

\gap

Finally, we consider the convex programming problem 
\beq \label{cp-cone}
f^*:= \min\{f(x): Ax-b \in \cK^*, x\in X\}, 
\eeq
for some $A\in \Re^{m \times n}$ and $b\in\Re^m$, where $f:X \to \Re$ is a smooth convex function 
whose gradient is  Lipschitz continuous gradient with constant $L_f>0$, $X \subseteq \Re^n$ is a 
closed convex set, and $\cK^*$ is the dual cone of a closed convex cone $\cK$.  

The Lagrangian dual function associated with \eqref{cp-cone} is given by 
\[
d(\mu) := \inf\{f(x)+\mu^T(Ax-b): x\in X\}, \ \ \ \forall \mu \in -\cK. 
\]

Assume that there exists a Lagrange multiplier for \eqref{cp-cone}, that is, a vector 
$\mu^* \in -\cK$ such that $d(\mu^*)=f^*$. Under this assumption, the following results 
are established in Corollary 2 and Proposition 10 of \cite{LaMo12}, respectively. 

\begin{lemma} \label{gap-infeas}
Let $\mu^*$ be a Lagrange multiplier for \eqref{cp-cone}. There holds: 
\[
f(x)-f^* \ge -\|\mu^*\|d_{\cK^*}(Ax-b), \  \ \ \forall x\in X.
\]
\end{lemma}

\begin{lemma} \label{approx-soln}
Let $\rho>0$ be given and $L_\rho = L_f+\rho\|A\|^2$. Consider the problem 
\beq \label{cp-pensub}
\Phi^*_\rho := \min\limits_{x\in X} \{\Phi_\rho(x) := f(x) + \frac{\rho}{2} [d_{\cK^*}(Ax-b)]^2\}.
\eeq
If $x\in X$ is a $\xi$-approximate solution of \eqref{cp-pensub}, i.e., 
$\Phi_\rho(x)-\Phi^*_\rho \le \xi$, then the pair $(x^+,\mu)$ defined 
as 
\[
\ba{lcl}
x^+ &:=& \Pi_X(x-\nabla \Phi_\rho(x)/{L_\rho}), \\ [5pt]
\mu &:=& \rho[Ax^+ - b - \Pi_{\cK^*}(Ax^+ - b)] 
\ea
\]
is in $X \times (-\cK)$ and satisfies $\mu^T\Pi_{\cK^*}(Ax^+ - b) = 0$
and the relations 
\[
\ba{lcl}
d_{\cK^*}(Ax^+-b) &\le& \frac{1}{\rho} \|\mu^*\| + \sqrt{\frac{\xi}{\rho}}, \\ [6pt]
\nabla f(x^+) + A^T \mu &\in& -\cN_X(x^+) + \cU(2\sqrt{2L_\rho \xi}),
\ea
\]
where $\mu^*$ is an arbitrary Lagrange multiplier for \eqref{cp-cone}.  
\end{lemma}

\section{$l_0$ regularized box constrained convex programming}
\label{l0-box}

In this section we consider a special case of \eqref{l0-cone}, that is, $l_0$ regularized box constrained 
convex programming problem  in the form of: 
\beq \label{l0-min}
\ba{rl}
\underline{F}^*  := \min & F(x) := f(x) + \lambda \|x\|_0 \\
\mbox{s.t.} & l \le x \le u
\ea
\eeq
for some $\lambda >0$, $l \in \bar\Re^n_{-}$ and $u \in \bar\Re^n_{+}$. Recently, Blumensath and 
Davies \cite{BlDa08,BlDa09} proposed an iterative hard thresholding (IHT) method for solving a special 
case of \eqref{l0-min} with $f(x)=\|Ax-b\|^2$, $l_i=-\infty$ and $u_i=\infty$ for all $i$. Our aim is 
to extend their IHT method to solve \eqref{l0-min} and study its convergence. In addition, we establish 
its iteration complexity for finding an $\epsilon$-local-optimal solution of \eqref{l0-min}. Finally, we 
propose a variant of the IHT method in which only ``local'' Lipschitz constant of $\nabla f$ is used. 

Throughout this section we assume that $f$ is a smooth convex function in $\cB$ whose gradient is 
Lipschitz continuous with constant $L_f>0$, and also that $f$ is bounded below on the set 
$\cB$, where
\beq \label{cB}
\cB := \{x\in \Re^n:l \le x \le u \}.
\eeq

We now present an IHT method for solving problem \eqref{l0-min}. 

\gap

\noindent
{\bf Iterative hard thresholding method for \eqnok{l0-min}:}  \\ [5pt]
Choose an arbitrary $x^0\in \cB$. Set $k=0$. 
\begin{itemize}
\item[1)] Solve the subproblem 
\beq \label{subprob}
x^{k+1} \in \Arg\min\limits_{x\in \cB}\{f(x^k)+\nabla f(x^k)^T(x-x^k)+\frac{L}{2}\|x-x^k\|^2+\lambda \|x\|_0\}.
\eeq
\item[2)]
Set $k \leftarrow k+1$ and go to step 1). 
\end{itemize}
\noindent
{\bf end}

\vgap

\begin{remark}
The subproblem \eqref{subprob} has a closed form solution  given in \eqref{IHT}.
\end{remark}

\vgap

In what follows, we study the convergence of the above IHT method for \eqref{l0-min}.  Before 
proceeding,  we introduce some notations that will be used subsequently. Define
\beqa
\cB_I &:=& \{x\in \cB: x_I = 0\}, \ \ \ \forall I \subseteq \{1,\ldots,n\}, \label{BI} \\ [6pt]
\Pi_\cB(x) &:=& \arg\min \{\|y-x\|: y \in \cB\}, \ \ \ \forall x \in \Re^n, \nn \\ [6pt]
s_L(x)  &:=&  x - \frac{1}{L} \nabla f(x), \ \ \ \forall x \in \cB, \label{sL} \\ [6pt]
I(x) &:=& \{i: x_i = 0\},  \ \ \ \forall x \in \Re^n \label{Ix}
\eeqa
for some constant $L > L_f$. 

The following lemma establishes some properties of  the operators $s_L(\cdot)$ and $\Pi_\cB(s_L(\cdot))$, 
 which will be used subsequently.

\begin{lemma} \label{diff-xy}
For any $x$, $y \in \Re^n$, there hold:
\bi
\item[\rm (1)]
$|[s_L(x)]^2_i-[s_L(y)]^2_i| \le 4(\|x-y\|+|[s_L(y)]_i|)\|x-y\|$;
\item[\rm (2)]
$|[\Pi_{\cB}(s_L(x))-s_L(x)]^2_i - [\Pi_{\cB}(s_L(y))-s_L(y)]^2_i| 
\le 4(\|x - y\| + |[\Pi_{\cB}(s_L(y))-s_L(y)]_i|)\|x - y\|$.
\ei
\end{lemma}

\begin{proof}
(1) We observe that 
\beqa
\|s_L(x)-s_L(y)\| &=& \|x-y-\frac1L(\nabla f(x)-\nabla f(y))\| \ \le \ \|x-y\| + 
\frac1L\|\nabla f(x)-\nabla f(y)\|, \nn \\ 
&\le &   (1+\frac{L_f}{L})  \|x-y\| \ \le \ 2\|x-y\|. \label{diff-s}
\eeqa
It follows from \eqref{diff-s}  that 
\[
\ba{lcl}
|[s_L(x)]^2_i-[s_L(y)]^2_i| & = & |[s_L(x)]_i+[s_L(y)]_i| \cdot |[s_L(x)]_i-[s_L(y)]_i|, \\ [7pt]
&\le & (|[s_L(x)]_i-[s_L(y)]_i| + 2|[s_L(y)]_i|) \cdot |[s_L(x)]_i-[s_L(y)]_i|, \\ [7pt]
& \le & 4(\|x-y\|+|[s_L(y)]_i|)\|x-y\|.   
\ea
\]

(2) It can be shown that
\[
\|\Pi_{\cB}(x)-x +y - \Pi_{\cB}(y)\| \ \le \ \|x-y\|.
\]
Using this inequality and \eqref{diff-s}, we then have 
\[
\ba{l} 
|[\Pi_{\cB}(s_L(x))-s_L(x)]^2_i - [\Pi_{\cB}(s_L(y))-s_L(y)]^2_i|  \\ [5pt]
 \le  (|[\Pi_{\cB}(s_L(x))-s_L(x)]_i - [\Pi_{\cB}(s_L(y))-s_L(y)]_i|
+2 |\Pi_{\cB}(s_L(y))-s_L(y)]_i|) \\ 
\ \ \ \cdot |[\Pi_{\cB}(s_L(x))-s_L(x)]_i - [\Pi_{\cB}(s_L(y))-s_L(y)]_i|, \\ [5pt]
 \le (\|s_L(x) - s_L(y)\| +2 |[\Pi_{\cB}(s_L(y))-s_L(y)]_i|)\cdot\|s_L(x) - s_L(y)\|, \\ [5pt]
 \le 4(\|x - y\| + |[\Pi_{\cB}(s_L(y))-s_L(y)]_i|)\|x - y\|.
\ea 
\]
\end{proof}

\gap

The following lemma shows that for the sequence $\{x^k\}$, the magnitude of any nonzero component 
$x^k_i$ cannot be too small for $k \ge 1$.

\begin{lemma}
Let $\{x^k\}$ be generated by the above IHT method. Then, for all $k \ge 0$,
\beq \label{lower-bdd}
|x^{k+1}_j|  \ \ge \ \delta := \min\limits_{i \notin I_0} \delta_i \ > \ 0, \ \ \ \mbox{if} \  \ x^{k+1}_j \neq 0,
\eeq
where $I_0 = \{i: l_i=u_i=0\}$ and 
\beq \label{deltai}
\delta_i = \left\{\ba{ll}
\min(u_i,\sqrt{2\lambda/L}), & \mbox{if} \ l_i=0, \\ [4pt]
\min(-l_i,\sqrt{2\lambda/L}), & \mbox{if} \ u_i=0, \\ [4pt]
\min(-l_i,u_i,\sqrt{2\lambda/L}), & \mbox{otherwise},
\ea\right. \quad\quad\quad\forall i \in I_0.
\eeq
\end{lemma}

\begin{proof}
One can observe from \eqref{subprob} that for $i=1, \ldots, n$, 
\beq \label{IHT}
x^{k+1}_i = \left\{\ba{ll}
[\Pi_\cB(s_L(x^k))]_i,  & \ \mbox{if} \ [s_L(x^k)]^2_i-[\Pi_{\cB}(s_L(x^k))-s_L(x^k)]^2_i > \frac{2\lambda}{L}, 
\\ [7pt]
0, &  \ \mbox{if} \ [s_L(x^k)]^2_i-[\Pi_{\cB}(s_L(x^k))-s_L(x^k)]^2_i < \frac{2\lambda}{L}, \\ [7pt]
[\Pi_\cB(s_L(x^k))]_i \ \mbox{or} \ 0,  & \ \mbox{otherwise}     
\ea\right.
\eeq
(see, for example, \cite{LuZh12}).  Suppose that $j$ is an index such that 
$x^{k+1}_j \neq 0$. Clearly, $j \notin I_0$, where $I_0$ is define above. It follows from \eqref{IHT} that 
\beq \label{xki}
x^{k+1}_j = [\Pi_\cB(s_L(x^k))]_j \neq  0, \quad\quad [s_L(x^k)]^2_j-[\Pi_{\cB}(s_L(x^k))-s_L(x^k)]^2_j \ge \frac{2\lambda}{L}.
\eeq
The second relation of \eqref{xki} implies that $|[s_L(x^k)]_j| \geq \sqrt{2\lambda/L} $. In addition, by 
the first relation of \eqref{xki} and the definition of $\Pi_\cB$, we have
\beq \label{xj}
x^{k+1}_j=[\Pi_\cB(s_L(x^k))]_j = \min(\max([s_L(x^k)]_j,l_j),u_j) \neq 0.
\eeq
Recall that $j \notin I_0$. We next show that $|x^{k+1}_j| \ge \delta_j$ by considering three separate 
cases: i) $l_j=0$; ii) $u_j=0$; and iii) $l_j u_j \neq 0$.  For case i), it follows from \eqref{xj} that 
$[s_L(x^k)]_j \ge 0$ and $x^{k+1}_j=\min([s_L(x^k)]_j,u_j)$. This together with the relation 
$|[s_L(x^k)]_j| \geq \sqrt{2\lambda/L} $ and the definition of $\delta_j$ implies that 
$|x^{k+1}_j| \ge \delta_j$.  By the similar arguments, we can show that $|x^{k+1}_j| \ge \delta_j$ 
also holds for the other two cases. Then, it is easy to see that the conclusion of this lemma holds.
\end{proof}

\gap

We next establish that the sequence $\{x^k\}$ converges to a local minimizer of $\eqref{l0-min}$, and 
moreover, $F(x^k)$ converges to a local minimum value of \eqref{l0-min}.

\begin{theorem} \label{limit-thm}
Let $\{x^k\}$ be generated by the above IHT method. Then, $x^k$ converges to a local minimizer 
$x^*$ of problem \eqref{l0-min} and moreover, $I(x^k) \to I(x^*)$, $\|x^k\|_0 \to \|x^*\|_0$ 
and $F(x^k) \to F(x^*)$.
\end{theorem}

\begin{proof}
Since $\nabla f$ is Lipschitz continuous with constant $L_f$, we have 
\[
f(x^{k+1}) \ \le \  f(x^k)+\nabla f(x^k)^T(x-x^k)+\frac{L_f}{2}\|x^{k+1}-x^k\|^2.
\] 
Using this inequality, the fact that $L > L_f$, and \eqref{subprob}, we obtain that 
\[
\ba{lcl}
F(x^{k+1}) &=& f(x^{k+1})+ \lambda \|x^{k+1}\|_0 \ \le \ \overbrace{ f(x^k)+\nabla 
f(x^k)^T(x^{k+1}-x^k)+\frac{L_f}{2}\|x^{k+1}-x^k\|^2 + \lambda \|x^{k+1}\|_0}^a, \nn \\ [14pt]
&\le & \underbrace{f(x^k)+\nabla f(x^k)^T(x^{k+1}-x^k)+\frac{L}{2}\|x^{k+1}-x^k\|^2 + 
\lambda \|x^{k+1}\|_0 }_b  \nn \\ [14pt]
&\le &  f(x^k)+\lambda \|x^k\|_0 \ = \ F(x^k), 
\ea
\]
where the last inequality follows from \eqref{subprob}.  The above inequality  implies that $\{F(x^k)\}$ 
is nonincreasing and moreover, 
\beq \label{diff-seq}
F(x^k) - F(x^{k+1}) \ \ge \ b-a \ = \  \frac{L-L_f}{2} \|x^{k+1}-x^k\|^2.
\eeq 
By the assumption, we know that $f$ is bounded below in $\cB$. It then follows that $\{F(x^k)\}$ is 
bounded below. Hence, $\{F(x^k)\}$ converges to a finite value as $k \to \infty$, which together with 
\eqref{diff-seq} implies that 
\beq \label{diff-lim}
\lim_{k \to \infty}\|x^{k+1}-x^k\| = 0.
\eeq
Let $I_k = I(x^k)$, where $I(\cdot)$ is defined in \eqref{Ix}.  In view of \eqref{lower-bdd}, we observe 
that 
\beq \label{x-change}
\|x^{k+1}-x^k\| \ge \delta \ \ \mbox{if} \ \ I_k \neq I_{k+1}.
\eeq
This together with \eqref{diff-lim} implies that $I_k$ does not change when $k$ is sufficient large. 
Hence, there exist some $K \ge 0$ and $I \subseteq \{1,\ldots,n\}$ such that  $I_k = I$ for all 
$k \geq K$. Then one can observe from \eqref{subprob} that 
\[
x^{k+1} = \arg\min\limits_{x\in \cB_I}\{f(x^k)+\nabla f(x^k)^T(x-x^k)+\frac{L}{2}\|x-x^k\|^2\}, \ \ \ 
\forall k > K,
\]
where $\cB_I$ is defined in \eqref{BI}. 
It follows from Lemma \ref{sd-lemma} that $x^k \to x^*$, where 
\beq \label{loc-min}
x^* \in \Arg\min \{f(x): x\in \cB_I\} .
\eeq  
 It is not hard to see from \eqref{loc-min} that $x^*$ is a local minimizer of \eqref{l0-min}. 
In addition, we know from \eqref{lower-bdd} that $|x^k_i| \ge \delta$ for 
$k > K$ and $i \notin I$. It yields $|x^*_i| \ge \delta$ for $i \notin I$ and $x^*_i = 0$ for $i \in I$. 
Hence, $I(x^k) = I(x^*)=I$ for all $k > K$, which clearly implies that $\|x^k\|_0 = \|x^*\|_0$ for every  
$k > K$. By continuity of $f$, we have $f(x^k) \to f(x^*)$. It then follows that 
\[
F(x^k) = f(x^k) + \lambda \|x^k\|_0 \to f(x^*)+\lambda \|x^*\|_0 = F(x^*).
\]
\end{proof}

\gap

As shown in Theorem \ref{limit-thm}, $x^k \to x^*$ for some local minimizer $x^*$ 
of \eqref{l0-min} and $F(x^k) \to F(x^*)$. Our next aim is to establish the iteration complexity 
of the IHT method for finding an $\epsilon$-local-optimal solution $x_\eps \in \cB$ of \eqref{l0-min}  
satisfying $F(x_\eps) \le F(x^*)+\epsilon$ and $I(x_\eps)=I(x^*)$. Before proceeding, we define 
\beqa
\alpha &=& \min\limits_{I \subseteq \{1,\ldots,n\}} \left\{\min\limits_{i} \left|[s_L(x^*)]^2_i
-[\Pi_{\cB}(s_L(x^*))-s_L(x^*)]^2_i -\frac{2\lambda}{L}\right|: \ x^* \in \Arg\min \{f(x): \ x\in \cB_I\}\right\}, \label{alpha} \\ [6pt]
\beta &=& \max\limits_{I \subseteq \{1,\ldots,n\}} \left\{\max\limits_{i} 
|[s_L(x^*)]_i|+|\Pi_{\cB}(s_L(x^*))-s_L(x^*)]_i|: \ x^* \in \Arg\min \{f(x): \ x\in \cB_I\}\right\}. \label{beta}
\eeqa

%
%
%
%

\begin{theorem} \label{complexity}
Assume that $f$ is a smooth strongly convex function with modulus $\sigma>0$. Suppose that 
$L>L_f$ is chosen such that $\alpha >0$.  Let $\{x^k\}$ be generated by the above IHT method, 
$I_k = I(x^k)$ for all $k$, $x^* =\lim_{k \to \infty} x^k$, and $F^* = F(x^*)$. Then, for any given 
$\eps>0$, the following statements hold:
\bi
\item[\rm (i)] The number changes of $I_k$ is at most 
$\left\lfloor \frac{2(F(x^0)-F^*)}{(L-L_f)\delta^2} \right\rfloor$. 
\item[\rm (ii)] The total number of iterations by the IHT method for finding an $\eps$-local-optimal 
solution $x_\eps \in \cB$ satisfying $I(x_\eps)=I(x^*)$ and $F(x_\eps) \le F^*+\eps$ is at most 
$2\lceil  L/\sigma\rceil \log \frac{\theta}{\eps}$,
where 
\beqa
&\theta = (F(x^0) - F^*)2^{\frac{\omega+3}{2}},  \quad 
\omega= \max\limits_t \left\{(d- 2 c) t -ct^2: 
0 \le t \le  \left\lfloor \frac{2(F(x^0)-F^*)}{(L-L_f)\delta^2} \right\rfloor\right\}, \label{thetaw}\\ [6pt]
&c=\frac{(L-L_f)\delta^2}{2(F(x^0)-\underline{F}^*)}, \quad\quad\ \gamma = 
\sigma(\sqrt{2\alpha+\beta^2}-\beta)^2/32, \label{gammac}\\ [6pt] 
&d = 2 \log(F(x^0) - \underline{F}^*) +4- 2 \log \gamma+ c.  \nn 
\eeqa
\ei
\end{theorem}

\begin{proof}
(i) As shown in Theorem \ref{limit-thm}, $I_k$ only changes for a finite number of times. 
Assume that $I_k$ only changes at $k=n_1+1, \ldots, n_J+1$, that is,   
\beq \label{chg-I}
I_{n_{j-1}+1} = \cdots = I_{n_j} \neq I_{n_j+1} = \cdots = I_{n_{j+1}}, \ j=1, \ldots, J-1,   
\eeq
where $n_0=0$. 

We next bound $J$, i.e., the  total number of changes of $I_k$. In view of \eqref{x-change} 
and \eqref{chg-I}, one can observe that
\[
\|x^{n_j+1}-x^{n_j}\| \ge \delta, \  \  \  j=1, \ldots, J,
\]
which together with \eqref{diff-seq} implies that
\beq \label{diff-Fx}
F(x^{n_j}) - F(x^{n_j+1}) \ge \frac12(L-L_f)\delta^2, \ \ \ j=1,\ldots, J.
\eeq
Summing up these inequalities and using the monotonicity of $\{F(x^k)\}$, we have
\beq \label{ineq-J}
\frac12(L-L_f)\delta^2 J \ \le \  F(x^{n_1}) - F(x^{n_J+1}) \ \le \ F(x^0)-F^*,
\eeq
and hence 
\beq \label{bound-J}
J \ \le \ \left\lfloor \frac{2(F(x^0)-F^*)}{(L-L_f)\delta^2} \right\rfloor.
\eeq

(ii) Let $n_j$ be defined as above for $j=1, \ldots, J$. We first show that  
\beq \label{diff-nj}
n_j - n_{j-1}  \ \le \ 2+ 2\lceil  L/\sigma\rceil 
\left\lceil \log\left(F(x^0) - (j-1)(L-L_f) \delta^2/2 - \underline{F}^*\right) - \log \gamma \right\rceil, 
\ \ \  j=1,\ldots, J,
\eeq 
where $\underline F^*$ and $\gamma$ are defined in \eqref{l0-min} and \eqref{gammac}, respectively.   
Indeed, one can observe from \eqref{subprob} that 
\[
x^{k+1} = \arg\min\limits_{x\in \cB}\{f(x^k)+\nabla f(x^k)^T(x-x^k)+\frac{L}{2}\|x-x^k\|^2: x_{I_{k+1}}=0\}.
\]
Therefore, for $j=1, \ldots, J$ and $k=n_{j-1}, \ldots, n_j-1$, 
\[
x^{k+1} = \arg\min\limits_{x\in \cB}\{f(x^k)+\nabla f(x^k)^T(x-x^k)+\frac{L}{2}\|x-x^k\|^2: x_{I_{n_j}}=0\}.
\]
We arbitrarily choose $1 \le j \le J$.
Let $\bx^*$ (depending on $j$) denote the optimal solution of
\beq \label{aux-prob}
\min\limits_{x\in \cB}\{f(x): x_{I_{n_j}}=0\}.
\eeq
One can observe that
\[
\|\bx^*\|_0 \ \le \  \|x^{n_{j-1}+1}\|_0.
\]
Also,  it follows from \eqref{diff-Fx} and  the monotonicity of $\{F(x^k)\}$  that 
\beq \label{bdd-Fx}
F(x^{n_j+1}) \ \le \ F(x^0) - \frac{j}{2}(L-L_f) \delta^2, \ \ \ j=1,\ldots, J.
\eeq
Using these relations and the fact that $F(\bx^*) \ge \underline{F}^*$,  we have
\beqa
f(x^{n_{j-1}+1}) - f(\bx^*) &=& F(x^{n_{j-1}+1}) - \lambda \|x^{n_{j-1}+1}\|_0 - F(\bx^*)+\lambda \|\bx^*\|_0, \nn \\ [6pt]
& \le & F(x^0) - \frac{j-1}{2}(L-L_f) \delta^2 - \underline{F}^*. \label{bdd-Fxj}
\eeqa
Suppose for a contradiction that \eqref{diff-nj} does not hold for some $1 \le j \le J$. Hence, we 
have 
\[
n_j - n_{j-1} > 2+ 2\lceil  L/\sigma\rceil 
\left\lceil \log\left(F(x^0) - (j-1)(L-L_f) \delta^2/2 - \underline{F}^*\right) - \log \gamma \right\rceil.
\]
This inequality and \eqref{bdd-Fxj} yields 
\[
n_j - n_{j-1} > 2+ 2\lceil  L/\sigma\rceil 
\left\lceil \log\frac{f(x^{n_{j-1}+1})-f(\bx^*)}{\gamma}\right\rceil.
\]
Using the strong convexity of $f$ and applying Theorem \ref{sd-thm} (ii) to \eqref{aux-prob} 
with $\epsilon=\gamma$, we obtain that
\[
\frac{\sigma}{2} \|x^{n_j}-\bx^*\|^2 \le f(x^{n_j}) - f(\bx^*) < \frac{\sigma}{32}(\sqrt{2\alpha+\beta^2}-\beta)^2.
\]
It implies that
\beq \label{xdiff-bdd}
\|x^{n_j}-\bx^*\| < \frac{\sqrt{2\alpha+\beta^2}-\beta}{4}.
\eeq
Using \eqref{xdiff-bdd}, Lemma \ref{diff-xy} and the definition of $\beta$, we have
\beqa
&& |[s_L(x^{n_j})]^2_i-[s_L(\bx^*)]^2_i - [\Pi_{\cB}(s_L(x^{n_j}))-s_L(x^{n_j})]^2_i + [\Pi_{\cB}(s_L(\bx^*))-s_L(\bx^*)]^2_i| \nn  \\ [5pt]
&& \le  |[s_L(x^{n_j})]^2_i-[s_L(\bx^*)]^2_i| + |[\Pi_{\cB}(s_L(x^{n_j}))-s_L(x^{n_j})]^2_i - [\Pi_{\cB}(s_L(\bx^*))-s_L(\bx^*)]^2_i| \nn \\ [5pt]
&&  \le \ 4(2\|x^{n_j}-\bx^*\|+\beta)\|x^{n_j}-\bx^*\| \ < \ \alpha, \label{diff-sp}
\eeqa
where the last inequality is due to \eqref{xdiff-bdd}. Let 
\[
I^* = \left\{i: [s_L(\bx^*)]^2_i-[\Pi_{\cB}(s_L(\bx^*))-s_L(\bx^*)]^2_i < \frac{2\lambda}{L}\right\}
\]
and let $\bar I^* = \{1,\ldots,n\}\setminus I^*$. Since $\alpha>0$, we know that 
\[
[s_L(\bx^*)]^2_i-[\Pi_{\cB}(s_L(\bx^*))-s_L(\bx^*)]^2_i > \frac{2\lambda}{L}, \ \ \ \forall i \in \bar I^*.
\]
It then follows from \eqref{diff-sp} and the definition of $\alpha$ that
\[
\ba{l}
[s_L(x^{n_j})]^2_i- [\Pi_{\cB}(s_L(x^{n_j}))-s_L(x^{n_j})]^2_i < \frac{2\lambda}{L},  \ \ \ \forall i \in I^*, \\ [6pt]
[s_L(x^{n_j})]^2_i- [\Pi_{\cB}(s_L(x^{n_j}))-s_L(x^{n_j})]^2_i > \frac{2\lambda}{L}, \ \ \ \forall i \in \bar I^*.
\ea
\]
Observe that $[\Pi_\cB(s_L(x^{n_j}))]_i \neq 0$ for all $i \in \bar I^*$. This fact together with \eqref{IHT} implies that
\[  
x^{n_j+1}_i =0, \ i \in I^* \ \ \mbox{and} \  \ x^{n_j+1}_i \neq 0, \ i \in \bar I^*.
\]
By a similar argument, one can show that 
\[
x^{n_j}_i =0, \ i \in I^* \  \ \mbox{and} \ \  x^{n_j}_i \neq 0, \ i \in \bar I^*.
\]
Hence, $I_{n_j} = I_{n_j+1} = I^*$, which is a contradiction to \eqref{chg-I}. We thus 
conclude that \eqref{diff-nj} holds.

Let $N_\eps$ denote the total number of iterations for finding an $\eps$-local-optimal solution 
$x_\eps \in \cB$ by the IHT method satisfying $I(x_\eps)=I(x^*)$ and $F(x_\eps) \le F^*+\eps$. We 
next establish an upper bound for $N_\eps$. Summing up the inequality \eqref{diff-nj} for 
$j=1,\ldots, J$,  we obtain that
\[
n_J \le  \sum^J\limits_{j=1} \left\{2+ 2\lceil  L/\sigma\rceil 
\left\lceil \log(F(x^0) - \frac{j-1}{2}(L-L_f) \delta^2 - \underline{F}^*) - \log \gamma \right\rceil\right\}.
\]
Using this inequality, \eqref{ineq-J},  and the facts that $L \ge \sigma$ and $\log(1-t) \le -t $ 
for all $t\in (0,1)$, we have 
 \beqa
 n_J &\le&  \sum^J\limits_{j=1} \left[2+ 2\lceil  L/\sigma \rceil 
\left(\log(F(x^0) - \frac{j-1}{2}(L-L_f) \delta^2 - \underline{F}^*) - \log \gamma +1\right)\right], \nn \\ 
& \le &   \sum^J\limits_{j=1} \left[2+ 2\lceil  L/\sigma\rceil 
\left(\log(F(x^0) - \underline{F}^*) - \frac{(L-L_f)\delta^2}{2(F(x^0)-\underline{F}^*)}(j-1) - \log \gamma +1\right)\right], \nn
\\ 
&\le& \lceil  L/\sigma \rceil\left[\underbrace{\left(2 \log(F(x^0) - \underline{F}^*) +4- 2 \log \gamma
+ \frac{(L-L_f)\delta^2}{2(F(x^0)-\underline{F}^*)}\right)}_d  J - \underbrace{\frac{(L-L_f)\delta^2}{2(F(x^0)-\underline{F}^*)}}_c J^2\right]. 
\label{n_J}
\eeqa
 By the definition of $n_J$, we observe that after $n_J+1$ iterations, the IHT method becomes the 
projected gradient method applied to the problem 
\[
x^* = \arg\min\limits_{x\in \cB}\{f(x): x_{I_{n_J+1}}=0\}.
\]
In addition, we know from Theorem \ref{limit-thm} that $I(x^k) = I(x^*)$ for all $k > n_J$. Hence, 
$f(x^k)-f(x^*)=F(x^k)-F^*$ when $k > n_J$. Using these facts and Theorem \ref{sd-thm} (ii),  
we have
\[
N_\eps \le n_J + 1+2\lceil  L/\sigma\rceil 
\left\lceil \log\frac{F(x^{n_J+1})-F^*}{\epsilon}\right\rceil.
\]
Using this inequality, \eqref{bdd-Fx},  \eqref{n_J} and the facts that $F^* \ge \underline{F}^*$, $L \ge \sigma$ and
 $\log(1-t) \le -t $ for all $t\in (0,1)$, we obtain that 
 \beqas
N_\eps &\le& n_J + 1+2\lceil  L/\sigma\rceil 
\left(\log(F(x^0) - \frac{J}{2}(L-L_f) \delta^2 - F^*) +1 - \log \eps\right), \\
&\le& n_J + \lceil  L/\sigma\rceil 
\left(2 \log(F(x^0) - F^*) - \frac{(L-L_f)\delta^2J}{F(x^0)-F^*} +3 -2 \log \eps\right) \\
&\le&\lceil  L/\sigma\rceil \left[ (d- 2 c) J -cJ^2 +2 \log(F(x^0) - F^*)+3 -2 \log \eps\right],
\eeqas
which together with \eqref{bound-J} and \eqref{thetaw}  implies that 
\[
N_\eps \ \le \ 2\lceil  L/\sigma\rceil \log \frac{\theta}{\eps}.
\]
\end{proof}

\gap

The iteration complexity given in Theorem \ref{complexity} is based on the assumption that $f$ is  
strongly convex in $\cB$. We next consider a case where $\cB$ is bounded and $f$ is convex but not 
strongly convex. We will establish the iteration complexity of finding an $\epsilon$-local-optimal 
solution  of \eqref{l0-min} by the IHT method applied to a perturbation of \eqref{l0-min}  obtained 
from adding a ``small'' strongly convex regularization term to $f$.

Consider a perturbation of \eqref{l0-min} in the form of
\beq \label{pert-l0}
{\underline{F}}^*_\nu := \min\limits_{x\in\cB} \{F_\nu(x) := f_\nu(x) + \lambda \|x\|_0\},
\eeq
where $\nu >0$ and 
\[
f_\nu(x) := f(x) + \frac{\nu}{2}\|x\|^2.
\]
One can easily see that $f_\nu$ is strongly convex in $\cB$ with modulus $\nu$ and moreover  
$\nabla f_\nu$ is Lipschitz continuous with constant $L_\nu$, where
\beq \label{Lnu}
L_\nu = L_f + \nu. 
\eeq

We next establish the iteration complexity of finding an $\epsilon$-local-optimal solution  of 
\eqref{l0-min} by the IHT method applied to \eqref{pert-l0}. Given any $L>0$, let $s_L$, 
$\alpha$ and $\beta$ be defined according to \eqref{sL}, \eqref{alpha} and \eqref{beta}, respectively, 
by replacing $f$ by $f_\nu$, and let $\delta$ be defined in \eqref{lower-bdd}.


\begin{theorem} 
Suppose that $\cB$ is bounded and $f$ is convex but not strongly convex. Let $\eps >0$ be 
arbitrarily given, $D = \max\{\|x\|: x \in \cB\}$, $\nu = \epsilon/D^2$, and $L > L_\nu$ be 
chosen such that $\alpha >0$. Let $\{x^k\}$ be generated by the IHT method applied to 
\eqref{pert-l0}, and let $x^* =\lim_{k \to \infty} x^k$, $F^*_\nu = F_\nu(x^*)$ and $F^* = \min\{F(x): 
x\in \cB_{I^*}\}$, where $I^*=\{i: x^*_i=0\}$.  Then, the total number of iterations by the IHT method 
for finding an $\eps$-local-optimal solution $x_\eps \in \cB$ satisfying $F(x_\eps) \le F^*+\eps$ 
is at most $2\left\lceil  \frac{D^2L_f}{\eps}+1\right\rceil \log \frac{2\theta}{\eps}$, where 
\beqas
&\theta = (F_\nu(x^0) - F_\nu^*)2^{\frac{\omega+3}{2}},  \quad 
\omega = \max\limits_t \left\{(d- 2 c) t -ct^2: 
0 \le t \le  \left\lfloor \frac{2(F_\nu(x^0)-F_\nu^*)}{(L-L_\nu)\delta^2} 
\right\rfloor\right\}, \\ [6pt]
&  c=\frac{(L-L_\nu)\delta^2}{2(F_\nu(x^0)-{\underline{F}}^*_\nu)}, \quad\quad \gamma = 
\nu(\sqrt{2\alpha+\beta^2}-\beta)^2/32, \\ [6pt]
& d = 2 \log(F_\nu(x^0) - {\underline{F}}^*_\nu) +4- 2 \log \gamma+ c.  
\eeqas
\end{theorem}

\begin{proof}
By Theorem \ref{complexity} (ii), we see that the IHT method applied to \eqref{pert-l0} finds an 
$\eps/2$-local-optimal solution $x_\eps \in \cB$ of \eqref{pert-l0}  satisfying $I(x_\eps)=I(x^*)$ and 
$F_\nu(x_\eps) \le F^*_\nu+\eps/2$ within $2\lceil L_\nu/\nu\rceil \log \frac{2\theta}{\eps}$ 
iterations. From the proof of Theorem \ref{limit-thm}, we observe that 
\[
 F_\nu(x^*) = \min \{F_\nu(x): x \in \cB_{I^*}\}.
\]
Hence, we have 
\[
F^*_\nu \ = \ F_\nu(x^*) \  \le \   \min\limits_{x\in\cB_{I^*}} f(x) + \frac{\nu D^2}{2} \ \le \  F^* + \frac{\eps}{2}.
\]
In addition, we observe that $F(x_\eps) \le F_\nu(x_\eps)$. Hence, it follows that 
\[
F(x_\eps) \ \le \ F_\nu(x_\eps) \ \le \ F^*_\nu+\frac{\eps}{2}  \ \le \ F^* + \eps.
\]
Note that $F^*$ is a local optimal value of \eqref{l0-min}. Hence, $x_\eps$ is an $\eps$-local-optimal 
solution of \eqref{l0-min}. The conclusion of this theorem then follows from \eqref{Lnu} and $\nu=\eps/D^2$.
\end{proof}
 
\gap

For the above IHT method, a fixed $L$ is used through all iterations, which may be too conservative.
To improve its practical performance, we can use ``local''  $L$  that is update dynamically. The resulting 
variant of the method is presented as follows. 

\gap

%

\noindent
{\bf A variant of IHT method for \eqref{l0-min}:}  \\ [5pt]
Let $0< L_{\min} < L_{\max}$, $\tau>1$ and $\eta>0$ be given. Choose an 
arbitrary $x^0 \in \cB$ and set $k=0$. 
\begin{itemize}
\item[1)] Choose $L^0_k \in [L_{\min}, L_{\max}]$ arbitrarily. Set $L_k = L^0_k$.  
\bi
\item[1a)] Solve the subproblem 
\beq \label{v-subprob}
x^{k+1} \in \Arg\min\limits_{x\in \cB} \{f(x^k)+\nabla f(x^k)^T(x-x^k)
+\frac{L_k}{2}\|x-x^k\|^2+\lambda \|x\|_0\}.
\eeq
\item[1b)] If 
\beq \label{descent}
F(x^k) - F(x^{k+1}) \ge \frac{\eta}{2} \|x^{k+1}-x^k\|^2
\eeq 
is satisfied,  then go to step 2). 
\item[1c)] Set $L_k \leftarrow \tau L_k$ and go to step 1a).
\ei
\item[2)]
Set $k \leftarrow k+1$ and go to step 1). 
\end{itemize}
\noindent
{\bf end}

\gap

\begin{remark}
$L^0_k$ can be chosen by the similar scheme as used in 
\cite{BaBo88,BiMaRa00}, that is, 
\[
L^0_k  = \max\left\{L_{\min},\min\left\{L_{\max},\frac{\Delta f^T \Delta x}{\|\Delta x\|^2}\right\}\right\}, \\ [6pt]
\]
where $\Delta x = x^k -x^{k-1}$ and $\Delta f= \nabla f(x^k) - \nabla f(x^{k-1})$.
\end{remark}

\gap

At each iteration, the IHT method solves a single subproblem in step 1). Nevertheless, its variant 
needs to solve a sequence of subproblems. We next show that for each outer iteration, its number of  
inner iterations is finite.  

\begin{theorem} \label{inner-iter}
For each $k \ge 0$, the inner termination criterion \eqref{descent} 
is satisfied after at most $\left\lceil \frac{\log(L_f+\eta)-\log(L_{\min})}
{\log \tau} +2\right\rceil$ inner iterations.
\end{theorem}

\begin{proof}
Let $\bar L_k$ denote the final value of $L_k$ at the $k$th outer iteration. By \eqref{v-subprob} 
and the similar arguments as for deriving \eqref{diff-seq}, one can show that 
\[
F(x^k) - F(x^{k+1}) \ \ge \  \frac{L_k-L_f}{2} \|x^{k+1}-x^k\|^2.
\]
Hence, \eqref{descent} holds whenever $L_k \ge L_f+\eta$, which together with 
the definition of $\bar L_k$ implies that $\bar L_k/\tau < L_f+\eta$, that is, $\bar L_k <\tau(L_f+\eta)$. 
Let $n_k$ denote the number of inner iterations for the $k$th outer iteration. Then, we have 
\[
L_{\min} \tau^{n_k-1} \le L^0_k \tau^{n_k-1} =  \bar L_k <  \tau(L_f+\eta).
\] 
Hence, $n_k \le \left\lceil \frac{\log(L_f+\eta)-\log(L_{\min})}{\log \tau} +2\right\rceil$ 
and the conclusion holds.
\end{proof}

%

\gap

We next establish that the sequence $\{x^k\}$ generated by the above variant of IHT method 
converges to a local minimizer of $\eqref{l0-min}$ and moreover, $F(x^k)$ converges to a local 
minimum value of \eqref{l0-min}. 

\begin{theorem} \label{outer-iter}
Let $\{x^k\}$ be generated by the above variant of IHT method. Then, $x^k$ converges to a local 
minimizer $x^*$ of problem \eqref{l0-min}, and moreover, $I(x^k) \to I(x^*)$, 
$\|x^k\|_0 \to \|x^*\|_0$ and $F(x^k) \to F(x^*)$.
\end{theorem}

\begin{proof}
Let $\bar L_k$ denote the final value of $L_k$ at the $k$th outer iteration. From the proof of Theorem 
\ref{inner-iter}, we know that $\bar L_k \in [L_{\min}, \tau(L_f+\eta))$. Using this fact and a similar 
argument as used to prove \eqref{lower-bdd}, one can obtain that 
\[
|x^{k+1}_i|  \ \ge \ \bar\delta := \min\limits_{i \notin I_0} \bar\delta_i \ > \ 0, \ \ \ \mbox{if} \  \ x^{k+1}_j \neq 0,
\]
where $I_0 = \{i: l_i=u_i=0\}$ and $\bar\delta_i$ is defined according to \eqref{deltai} by replacing 
$L$ by $\tau(L_f+\eta)$ for all $i \in I_0$. It implies that 
\[
\|x^{k+1}-x^k\| \ge \bar \delta \ \ \mbox{if} \ \ I(x^k) \neq I(x^{k+1}).
\]
The conclusion then follows from this inequality and the similar arguments as used in 
the proof of Theorem \ref{limit-thm}.
\end{proof}

\section{$l_0$-regularized convex cone programming}
\label{l0-cp}


In this section we consider $l_0$-regularized convex cone programming problem \eqref{l0-cone} and 
propose IHT methods for solving it. In particular, we apply the IHT method proposed in Section 
\ref{l0-box} to a quadratic penalty relaxation of \eqref{l0-cone} and establish the iteration 
complexity for finding an $\eps$-approximate local minimizer of \eqref{l0-cone}. We also 
propose a variant of the method in which the associated penalty parameter is dynamically 
updated, and show that every accumulation point is a local minimizer of \eqref{l0-cone}.

Let $\cB$ be defined in \eqref{cB}. We assume that $f$ is a smooth convex function in $\cB$, 
$\nabla f$ is Lipschitz continuous with constant $L_f$ and that $f$ is bounded below on $\cB$. In 
addition,  we make the following assumption throughout this section.

\begin{assumption} \label{assump-cone}
For each $I \subseteq \{1,\ldots,n\}$, there exists a Lagrange multiplier for 
\beq \label{convex-subprob}
 f^*_I = \min\{f(x): A x - b \in \cK^*, x \in \cB_I\},  
\eeq
provided that \eqref{convex-subprob} is feasible, that is, there exists $\mu^* \in -\cK$ 
such that $f^*_I = d_I(\mu^*)$, 
where 
\[
d_I(\mu) := \inf\{f(x)+ \mu^T(Ax-b):  x\in \cB_I\}, \ \forall \mu \in -\cK. 
\]
\end{assumption}

Let $x^*$ be a point in $\cB$, and let $I^*=\{i: x^*_i=0\}$. One can observe that $x^*$ is a 
local minimizer of \eqref{l0-cone} if and only if $x^*$ is a minimizer of \eqref{convex-subprob} 
with $I = I^*$. Then, in view of Assumption \ref{assump-cone}, we see that $x^*$ is a local 
minimizer of \eqref{l0-cone} if and only if $x^* \in \cB$ and there exists $\mu^* \in -\cK$ such 
that 
\beq \label{KKT-cond}
\ba{l}
Ax^* - b \in \cK^*, \ \ \ (\mu^*)^T(Ax^*-b)=0, \\ [5pt]
\nabla f(x^*) + A^T \mu^* \in -\cN_{\cB_{I^*}}(x^*).
\ea
\eeq

Based on the above observation, we can define an approximate local minimizer of \eqref{l0-cone} to 
be the one that nearly satisfies \eqref{KKT-cond}.

\begin{definition} 
Let $x^*$ be a point in $\cB$, and let $I^*=\{i: x^*_i=0\}$. 
$x^*$ is an $\epsilon$-approximate local minimizer of \eqref{l0-cone} if there exists 
$\mu^* \in -\cK$ such that 
\[
\ba{l}
d_{\cK^*}(A x^* - b) \ \le \ \epsilon, \ \ \ (\mu^*)^T \Pi_{\cK^*}(Ax^*-b) = 0, \\ [6pt]
\nabla f(x^*) + A^T \mu^* \in -\cN_{\cB_{I^*}}(x^*) + \cU(\epsilon).
\ea
\]  
\end{definition}

\gap

In what follows,  we propose an IHT method for finding an approximate local minimizer of 
\eqref{l0-cone}. In particular, we apply the IHT method or its variant to a quadratic 
penalty relaxation of \eqref{l0-cone} which is in the form of 
\beq \label{l0-penalty}
{\underline\Psi}^*_\rho := \min\limits_{x \in \cB} \{\Psi_\rho(x) := \Phi_\rho(x) + \lambda \|x\|_0\},
\eeq
where 
\beq \label{Phi-rho}
\Phi_\rho(x) := f(x) + \frac{\rho}{2} [d_{\cK^*}(Ax-b)]^2
\eeq

It is not hard to show that the function $\Phi_\rho$ is convex differentiable and moreover 
$\nabla \Phi_{\rho}$ is Lipschitz continuous with constant 
\beq \label{L-rho}
L_\rho = L_f + \rho\|A\|^2
\eeq
(see, for example, Proposition 8 and Corollary 9 of \cite{LaMo12}). Therefore, problem 
\eqref{l0-penalty}  can be suitably solved by the IHT method or its variant proposed in Section 
\ref{l0-box}. 

Under the assumption that $f$ is strongly convex in $\cB$,  we next establish the iteration complexity 
of the IHT method applied to  \eqref{l0-penalty} for finding an approximate local minimizer of 
\eqref{l0-cone}. Given any $L>0$, let $s_L$, $\alpha$ and $\beta$ be defined according to 
\eqref{sL}, \eqref{alpha} and \eqref{beta}, respectively, by replacing $f$ by $\Phi_\rho$, and let 
$\delta$ be defined in \eqref{lower-bdd}.

\begin{theorem}
Assume that $f$ is a smooth strongly convex function with modulus $\sigma>0$. Given any 
$\eps >0$, let
\beq \label{rho}
\rho = \frac{t}{\epsilon}+\frac{1}{\sqrt{8}\|A\|}
\eeq
for any $t \ge \max\limits_{I \subseteq \{1,\ldots,n\}}\min\limits_{\mu \in \Lambda_I}\|\mu\|$, where  
$\Lambda_I$ is the set of Lagrange multipliers of \eqref{convex-subprob}.  Let $L > L_\rho$ be chosen 
such that $\alpha >0$. Let $\{x^k\}$ be generated by the IHT method applied to \eqref{l0-penalty}, 
and let $x^* =\lim_{k \to \infty} x^k$ and $\Psi^*_\rho = \Psi_\rho(x^*)$.  Then the IHT method finds 
an $\epsilon$-approximate local minimizer of \eqref{l0-cone} in at most 
\[
N :=2\left\lceil \frac{L_\rho}{\sigma} \right\rceil \log \frac{8L_\rho\theta}{\eps^2}
\]
iterations, where 
\beqas
&\theta = (\Psi_\rho(x^0) - \Psi_\rho^*)2^{\frac{\omega+3}{2}},  \quad 
\omega = \max\limits_t \left\{(d- 2 c) t -ct^2: 
0 \le t \le  \left\lfloor \frac{2(\Psi_\rho(x^0)-\Psi_\rho^*)}{(L-L_\rho)\delta^2} 
\right\rfloor\right\}, \\ [6pt]
& c=\frac{(L-L_\rho)\delta^2}{2(\Psi_\rho(x^0)-{\underline\Psi}^*_\rho)}, \quad\quad  
\gamma = \sigma(\sqrt{2\alpha+\beta^2}-\beta)^2/32, \\ [6pt]
&d = 2 \log(\Psi_\rho(x^0) - {\underline\Psi}^*_\rho) +4- 2 \log \gamma+ c.  
\eeqas
\end{theorem}

\begin{proof}
We know from Theorem \ref{limit-thm} that $x^k \to x^*$ for some local minimizer $x^*$ 
of \eqref{l0-penalty}, $I(x^k) \to I(x^*)$ and $\Psi_\rho(x^k) \to \Psi_\rho(x^*) = \Psi^*_\rho$. 
By Theorem \ref{complexity}, after at most $N$ iterations, the IHT method generates $\tx \in \cB$ such at $I(\tx)=I(x^*)$ and 
$\Psi_\rho(\tx)-\Psi_\rho(x^*) \le \xi := \epsilon^2/(8L_\rho)$. It then follows that $\Phi_\rho(\tx)-\Phi_\rho(x^*) \le \xi$.
 Since $x^*$ is a local minimizer of \eqref{l0-penalty}, we observe that 
\beq \label{penalty-subprob}
x^* = \arg\min_{x\in \cB_{I^*}} \Phi_\rho(x),
\eeq
where $I^*=I(x^*)$. Hence, $\tx$ is a $\xi$-approximate solution of \eqref{penalty-subprob}. 
Let $\mu^* \in \Arg\min\{\|\mu\|: \mu \in \Lambda_{I^*}\}$, where  $\Lambda_{I^*}$ is the set of 
Lagrange multipliers of \eqref{convex-subprob} with $I=I^*$. In view of Lemma \ref{approx-soln}, we 
see that the pair $(\tx^+,\mu)$ defined as 
$\tx^+ := \Pi_{\cB_{I^*}}(\tx-\nabla \Phi_\rho(\tx)/{L_\rho})$ and $\mu := \rho[A\tx^+ - b - \Pi_{\cK^*}(A\tx^+ - b)]$  
satisfies 
\[
\ba{l}
\nabla f(\tx^+) + A^T \mu \ \in \ -\cN_{\cB_{I^*}}(\tx^+) + \cU(2\sqrt{2L_{\rho}\xi}) \ = \-\cN_{\cB_I}(\tx^+) + \cU(\epsilon), \\ [6pt]
d_{\cK^*}(A\tx^+-b) \ \le \ \frac{1}{\rho}\|\mu^*\| + \sqrt{\frac{\xi}{\rho}} \ \le \ 
\frac{1}{\rho}\left(\|\mu^*\|+\frac{\epsilon}{\sqrt8 \|A\|}\right) 
\ \le \ \epsilon,
\ea
\]
where the last inequality is due to \eqref{rho} and the assumption $t \ge \hat t \ge \|\mu^*\|$. 
Hence, $\tx^+$ is an $\epsilon$-approximate local minimizer of \eqref{l0-cone}. 
\end{proof}

\gap

We next consider finding an $\epsilon$-approximate local minimizer of \eqref{l0-cone} for the 
case where $\cB$ is bounded and $f$ is convex but not strongly convex. In particular, we  apply 
the IHT method or its variant to a quadratic penalty relaxation of a perturbation of \eqref{l0-cone} 
obtained from adding a ``small' ' strongly convex regularization term to $f$.

Consider a perturbation of \eqref{l0-cone} in the form of
\beq \label{pert-l0-cone}
\min\limits_{x\in\cB}\{f(x) + \frac{\nu}{2}\|x\|^2 +\lambda \|x\|_0: \ A x - b \in \cK^*\}.
\eeq

The associated quadratic penalty problem for \eqref{pert-l0-cone} is given by
\beq \label{qp-pert}
{\underline\Psi}^*_{\rho,\nu} := \min\limits_{x\in\cB} \{\Psi_{\rho,\nu}(x) := \Phi_{\rho,\nu}(x) + \lambda \|x\|_0\},
\eeq 
where 
\[
\Phi_{\rho,\nu}(x) := f(x) + \frac{\nu}{2}\|x\|^2 + \frac{\rho}{2} [d_{\cK^*}(Ax-b)]^2.
\]
One can easily see that $\Phi_{\rho,\nu}$ is strongly convex in $\cB$ with modulus $\nu$ and moreover 
$\nabla \Phi_{\rho,\nu}$ is Lipschitz continuous with constant 
\[
L_{\rho,\nu} := L_f + \rho \|A\|^2 + \nu. 
\]
Clearly, the IHT method or its variant can be suitably applied to \eqref{qp-pert}. We next establish 
the iteration complexity of the IHT method applied to \eqref{qp-pert} for finding an approximate local 
minimizer of \eqref{l0-cone}. Given any $L>0$, let $s_L$, $\alpha$ and $\beta$ be defined according to 
\eqref{sL}, \eqref{alpha} and \eqref{beta}, respectively, by replacing $f$ by $\Phi_{\rho,\nu}$, and let 
$\delta$ be defined in \eqref{lower-bdd}.

\begin{theorem} 
Suppose that $\cB$ is bounded and $f$ is convex but not strongly convex. Let $\eps >0$ be 
arbitrarily given, $D = \max\{\|x\|: x \in \cB\}$, 
\beq \label{rho-nu}
\rho = \frac{\left(\sqrt{D}+\sqrt{D+16t + \frac{2\sqrt{2}\epsilon}
{\|A\|}}\right)^2}{16\epsilon}, \ \ \ \nu = \frac{\epsilon}{2D}
\eeq
for any $t \ge \max\limits_{I \subseteq \{1,\ldots,n\}}\min\limits_{\mu \in \Lambda_I}\|\mu\|$, where  
$\Lambda_I$ is the set of Lagrange multipliers of \eqref{convex-subprob}. Let $L > L_{\rho,\nu}$ be 
chosen such that $\alpha >0$. Let $\{x^k\}$ be generated by the IHT method applied to 
\eqref{qp-pert}, and let $x^* =\lim_{k \to \infty} x^k$ and $\Psi^*_{\rho,\nu} = \Psi_{\rho,\nu}(x^*)$.  
Then the IHT method finds an $\epsilon$-approximate local minimizer of \eqref{l0-cone} in at most 
\[
N:= 2\left\lceil \frac{2DL_{\rho,\nu}}{\eps} \right\rceil \log \frac{32L_{\rho,\nu}\theta}{\eps^2}
\]
iterations, where 
\beqas
&\theta = (\Psi_{\rho,\nu}(x^0) - \Psi_{\rho,\nu}^*)2^{\frac{\omega+3}{2}},  \quad 
\omega = \max\limits_t \left\{(d- 2 c) t -ct^2: 
0 \le t \le  \left\lfloor \frac{2(\Psi_{\rho,\nu}(x^0)-\Psi_{\rho,\nu}^*)}{(L-L_{\rho,\nu})\delta^2} \right\rfloor\right\}, \\ [6pt]
& c=\frac{(L-L_{\rho,\nu})\delta^2}{2(\Psi_{\rho,\nu}(x^0)-{\underline\Psi}^*_{\rho,\nu})}, \quad\quad 
\gamma = \nu(\sqrt{2\alpha_{\rho,\nu}+\beta^2_{\rho,\nu}}-\beta_{\rho,\nu})^2/32, \\ [6pt] 
& d = 2 \log(\Psi_{\rho,\nu}(x^0) - {\underline\Psi}^*_{\rho,\nu}) +4- 2 \log \gamma+ c.  
\eeqas
\end{theorem}

\begin{proof}
From Theorem \ref{limit-thm}, we know that $x^k \to x^*$ for some local minimizer $x^*$ 
of \eqref{qp-pert}, $I(x^k) \to I(x^*)$ and $\Psi_{\rho,\nu}(x^k) \to \Psi_{\rho,\nu}(x^*) = \Psi^*_{\rho,\nu}$. 
By Theorem \ref{complexity}, after at most $N$ iterations, the IHT method applied to \eqref{qp-pert} 
generates $\tx \in \cB$ such at $I(\tx)=I(x^*)$ and 
$\Psi_{\rho,\nu}(\tx)-\Psi_{\rho,\nu}(x^*) \le \xi := \epsilon^2/(32L_{\rho,\nu})$. It then follows that 
$\Phi_{\rho,\nu}(\tx)-\Phi_{\rho,\nu}(x^*) \le \xi$. Since $x^*$ is a local minimizer of \eqref{qp-pert}, we see 
that 
\beq \label{pert-subprob}
x^* = \arg\min_{x\in \cB_{I^*}} \Phi_{\rho,\nu}(x),
\eeq
where $I^*=I(x^*)$. Hence, $\tx$ is a $\xi$-approximate solution of \eqref{pert-subprob}. 
In view of Lemma \ref{approx-soln}, we see that the pair $(\tx^+,\mu)$ defined as 
$\tx^+ := \Pi_{\cB_I}(\tx-\nabla \Phi_{\rho,\nu}(\tx)/{L_{\rho,\nu}})$ and 
$\mu := \rho[A\tx^+ - b - \Pi_{\cK^*}(A\tx^+ - b)]$ satisfies  
\[
\nabla f(\tx^+) + \nu \tx^+ +  A^T \mu \ \in \ -\cN_{\cB_{I^*}}(\tx^+) + \cU(2\sqrt{2L_{\rho,\nu}\xi}) 
\ = \ -\cN_{\cB_{I^*}}(\tx^+) + \cU(\epsilon/2),  
\]
which together with the fact that $\nu\|\tx^+\| \le \nu D \le \epsilon/2$ implies that 
\[
\nabla f(\tx^+) + A^T \mu \ \in \ -\nu \tx^+ -\cN_{\cB_{I^*}}(\tx^+) + \cU(\epsilon/2) \ 
\subseteq \ -\cN_{\cB_{I^*}}(\tx^+) + \cU(\epsilon).
\]
In addition, it follows from Lemma \ref{proj-grad} (c) that $\Phi_{\rho,\nu}(\tx^+) \le \Phi_{\rho,\nu}(\tx)$, and hence 
\[
\Phi_{\rho,\nu}(\tx^+) -  \Phi_{\rho,\nu}(x^*) \ \le \ \Phi_{\rho,\nu}(\tx)-\Phi_{\rho,\nu}(x^*) \ \le \ 
\xi.
\]
Let $\Phi^*_\rho = \min\{\Phi_\rho(x): x\in \cB_{I^*}\}$, where $\Phi_\rho$ is defined in 
\eqref{Phi-rho}.  Notice that $\Phi_{\rho,\nu}(x^*) \le \Phi^*_\rho + \nu D^2/2$. It then 
follows that 
\[
\Phi_{\rho}(\tx^+) -  \Phi^*_{\rho} \ \le \ \Phi_{\rho,\nu}(\tx^+) - \Phi_{\rho,\nu}(x^*) + \frac{\nu D^2}{2} 
\ \le \ \xi + \frac{\epsilon D}{4} \ \le \ \frac{\epsilon^2}{32\rho\|A\|^2} + \frac{\epsilon D}{4}. 
\]
Let $\mu^* \in \Arg\min\{\|\mu\|: \mu \in \Lambda_{I^*}\}$, where  $\Lambda_{I^*}$ is the set of 
Lagrange multipliers of \eqref{convex-subprob} with $I=I^*$. 
In view of Lemma \ref{approx-soln} and the assumption $t \ge \hat t \ge \|\mu^*\|$, we obtain that 
\[
d_{\cK^*}(A\tx^+-b) \ \le \ \frac{1}{\rho}\|\mu^*\| + \sqrt{\frac{\epsilon^2}{32\rho^2\|A\|^2}+\frac{\epsilon D}{4\rho}} 
\ \le \ \frac{1}{\rho}\left(t+\frac{\epsilon}{\sqrt{32}\|A\|}\right)+\sqrt{\frac{\epsilon D}{4\rho}} 
\ = \ \epsilon,
\]
where the last inequality is due to \eqref{rho-nu}. Hence, $\tx^+$ is an $\epsilon$-approximate 
local minimizer of \eqref{l0-cone}. 
\end{proof}

\gap

For the above method, the fixed penalty parameter $\rho$ is used through all iterations, which may be 
too conservative. To improve its practical performance, we can update $\rho$ dynamically. The 
resulting variant of the method is presented as follows. Before proceeding, we define the 
projected gradient of $\Phi_\rho$ at $x \in \cB_I$ with respect to $\cB_I$ as  
\beq \label{g-rho}
g(x;\rho,I) = L_\rho[x-\Pi_{\cB_I}(x-\frac{1}{L_\rho}\nabla \Phi_\rho(x))], 
\eeq
where $I \subseteq \{1,\ldots,n\}$, and $\Phi_\rho$ and $L_\rho$ are defined in \eqref{Phi-rho} and 
\eqref{L-rho}, respectively.

\gap

\noindent
{\bf A variant of IHT method for \eqnok{l0-cone}:}  \\ [5pt]
Let $\{\epsilon_k\}$ be a positive decreasing sequence. Let $\rho_0 >0$, 
$\tau > 1$, $t > \max\limits_{I \subseteq \{1,\ldots,n\}}\min\limits_{\mu \in \Lambda_I}\|\mu\|$, 
where $\Lambda_I$ is the set of Lagrange multipliers of \eqref{convex-subprob}. 
Choose an arbitrary $x^0\in \cB$. Set $k=0$. 
\begin{itemize}
\item[1)] Start from $x^{k-1}$ and apply the IHT method or its variant  to problem 
\eqref{l0-penalty} with $\rho=\rho_k$ until finding some $x^k \in \cB$ such that 
\beq \label{inner-cond-c}
d_{\cK^*}(Ax^k-b) \le \frac{t}{\rho_k}, \ \ \ \ \ \|g(x^k;\rho_k,I_k)\|  \le  
\min\{1,L_{\rho_k}\}\epsilon_k, 
\eeq
where $I_k = I(x^k)$.
\item[2)]
Set $\rho_{k+1} := \tau\rho_k$.
\item[3)]
Set $k \leftarrow k+1$ and go to step 1). 
\end{itemize}
\noindent
{\bf end}

\vgap


The following theorem shows that $x^k$ satisfying \eqref{inner-cond-c} can be found 
within a finite number of iterations by  the IHT method or its variant  applied to problem 
\eqref{l0-penalty} with $\rho=\rho_k$. Without loss of generality, we consider the IHT 
method or its variant applied to problem \eqnok{l0-penalty} with any given $\rho >0$.

\begin{theorem} \label{inner}
Let $x_0 \in \cB$ be an arbitrary point and the sequence $\{x_l\}$ be 
generated by the IHT method or its variant applied to problem \eqnok{l0-penalty}.
Then, the following statements hold:
\bi
\item[\rm(i)]
$\lim\limits_{l\to\infty} g(x_l;\rho,I_l) = 0$, where $I_l = I(x_l)$ for all $l$.  
\item[\rm(ii)]
$\lim\limits_{l\to\infty} d_{\cK^*}(Ax_l-b) \le  \frac{\hat t}{\rho}$, where 
$\hat t :=  \max\limits_{I \subseteq \{1,\ldots,n\}}
\min\limits_{\mu \in \Lambda_I}\|\mu\|$ and $\Lambda_I$ is the set of Lagrange multipliers 
of \eqref{convex-subprob}.
\ei
\end{theorem}

\begin{proof}
(i) It follows from Theorems \ref{limit-thm} and \ref{outer-iter} that $x_l \to x^*$ for some local 
minimizer $x^*$ of \eqref{l0-penalty} and moreover, $\Phi_\rho(x_l) \to \Phi_\rho(x^*)$ and 
$I_l \to I^*$, where $I_l=I(x_l)$ and $I^*=I(x^*)$. We also know that  
\[
x^* \in \Arg\min\limits_{x\in\cB_{I^*}} \Phi_\rho(x).
\]
It then follows from Lemma \ref{proj-grad} (d) that 
\[
\Phi_\rho(x_l) -\Phi_\rho(x^*) \ \ge \ \frac{1}{2L_\rho} \|g(x_l;\rho,I^*)\|^2, \ \ \ l \ge N. 
\]
Using this inequality and  $\Phi_\rho(x_l) \to \Phi_\rho(x^*)$, we thus have $g(x_l;\rho,I^*) \to 0$. 
Since $I_l=I^*$ for $l \ge N$, we also have $g(x_l;\rho,I_l) \to 0$. 

(ii) Let $f^*_I$ be defined in \eqref{convex-subprob}. Applying Lemma \ref{gap-infeas} to 
problem \eqref{convex-subprob}, we know that 
\beq \label{feas-bdd}
f(x_l) - f^*_{I(l)}  \ \ge \ -\hat t d_{\cK^*}(Ax_l-b), \ \ \  \forall l,
\eeq
where $\hat t$ is defined above. Let $x^*$ and $I^*$ be defined in the proof of statement (i). 
We observe that $f^*_{I^*} \ge \Phi_\rho(x^*)$. Using this relation and 
\eqref{feas-bdd}, we have that for sufficiently large $l$, 
\[
\ba{lcl}
\Phi_\rho(x_l) - \Phi_\rho(x^*) &=& f(x_l) + \frac{\rho}{2}[d_{\cK^*}(Ax_l-b)]^2 - \Phi_\rho(x^*) 
\ \ge \ f(x_l) - f^*_{I^*} + \frac{\rho}{2}[d_{\cK^*}(Ax_l-b)]^2 \\ [6pt] 
&=&  f(x_l) - f^*_{I(l)} + \frac{\rho}{2}[d_{\cK^*}(Ax_l-b)]^2 \ge \ -\hat t d_{\cK^*}(Ax_l-b) 
+ \frac{\rho}{2}[d_{\cK^*}(Ax_l-b)]^2,  
\ea
\]
which implies that 
\[
d_{\cK^*}(Ax_l-b) \ \le \ \frac{\hat t}{\rho} + \sqrt{\frac{\Phi_\rho(x_l) - \Phi_\rho(x^*)}{\rho}}.
\]
This inequality together with the fact $\lim_{l\to\infty} \Phi_\rho(x_l) = \Phi_\rho(x^*)$ 
yields statement (ii).
\end{proof}

\gap

\begin{remark}
From Theorem \ref{inner}, we can see that the inner iterations of the above method terminates finitely. 
\end{remark}

\gap

We next establish convergence of the outer iterations of the above variant of  the IHT method for 
\eqref{l0-cone}. In particular, we show that every accumulation point of $\{x^k\}$ is a local minimizer 
of \eqref{l0-cone}. 

\begin{theorem}
Let $\{x^k\}$ be the sequence generated by the above variant of the IHT method for solving 
\eqref{l0-cone}. Then any accumulation point of $\{x^k\}$ is a local minimizer of \eqref{l0-cone}. 
\end{theorem}

\begin{proof}
Let 
\[
\tx^k = \Pi_{\cB_{I_k}}(x^k-\frac{1}{L_{\rho_k}} \nabla \Phi_{\rho_k}(x^k)).
\]
Since $\{x^k\}$ satisfies \eqref{inner-cond-c}, it follows from Lemma \ref{proj-grad} (a) 
that 
\beq \label{opt-cond-seq}
\nabla \Phi_{\rho_k}(x^k) \ \in \ -\cN_{\cB_{I_k}}(\tx^k) + \cU(\epsilon_k),
\eeq
where $I_k=I(x^k)$. Let $x^*$ be any accumulation point of $\{x^k\}$. Then there exists a 
subsequence $K$ such that $\{x^k\}_K \to x^*$. By passing to a subsequence if necessary, 
we can assume that $I_k = I$ for all $k\in K$. Let 
\[
\mu^k = \rho_k [Ax^k-b-\Pi_{\cK^*}(Ax^k-b)].
\]
We clearly see that 
\beq \label{orth}
(\mu^k)^T\Pi_{\cK^*}(Ax^k-b)=0.
\eeq
Using \eqref{opt-cond-seq} and the definitions of $\Phi_\rho$ and $\mu^k$, 
we have
\beq \label{approx-opt}
\nabla f(x^k) + A^T\mu^k \in -\cN_{\cB_{I}}(\tx^k) + \cU(\epsilon_k), \ \ \ \forall k\in\cK.
\eeq 
By \eqref{g-rho}, \eqref{inner-cond-c} and the definition of $\tx^k$, one can observe that 
\beq \label{diff-x}
\|\tx^k - x^k\| \ = \ \frac{1}{L_{\rho_k}} \|g(x^k;\rho_k,I_k)\| \ \le \ \epsilon_k.
\eeq
In addition, notice that $\|\mu^k\| = \rho_k d_{\cK^*}(Ax^k-b)$, which together 
with \eqref{inner-cond-c} implies that $\|\mu^k\| \le t$ for all $k$. Hence, $\{\mu^k\}$ 
is bounded. By passing to a subsequence if necessary, we can assume that $\{\mu^k\}_K \to \mu^*$. 
Using \eqref{diff-x} and upon taking limits on both sides of \eqnok{orth} and 
\eqnok{approx-opt} as $k\in K \to \infty$, we have
\[
(\mu^*)^T\Pi_{\cK^*}(Ax^*-b)=0, \ \ \ \nabla f(x^*) + A^T\mu^* \in  -\cN_{\cB_{I}}(x^*)
\]
In addition, since $x^k_I =0$ for $k\in K$, we know that $x^*_I=0$. Also, it follows 
from \eqref{inner-cond-c} that $d_{\cK^*}(Ax^*-b)=0$, which implies that 
$Ax^*-b \in \cK^*$. These relations yield that 
\[
x^* \in \Arg\min\limits_{x\in \cB_I}\{f(x): Ax-b \in \cK^*\},
\]
and hence, $x^*$ is a local minimizer of \eqref{l0-cone}.
\end{proof}

\section{Concluding remarks}
\label{conclude}

In this paper we studied iterative hard thresholding (IHT) methods for solving $l_0$ regularized  
convex cone programming problems. In particular, we first proposed  an IHT method and its variant 
for solving $l_0$ regularized box constrained convex programming. We showed that the sequence 
generated by these methods converges to a local minimizer. Also, we established the iteration 
complexity of the IHT method for finding an $\eps$-local-optimal solution. We then proposed a 
method for solving $l_0$ regularized  convex cone programming by applying the IHT method to its 
quadratic penalty relaxation and established its iteration complexity for finding an $\eps$-approximate 
local minimizer. Finally, we proposed a variant of this method in which the associated penalty 
parameter is dynamically updated, and showed that every accumulation point is a local minimizer 
of the problem.

Some of the methods studied in this paper can be extended to solve some $l_0$ regularized nonconvex 
optimization problems.  For example,  the  IHT method and its variant can be applied to problem 
\eqref{l0-min} in which $f$ is nonconvex and $\nabla f$ is Lipschitz continuous. In addition,  the 
numerical study of the IHT methods will be presented in the working paper \cite{HuLiLu12}. Finally, 
it would be interesting to extend the methods of this paper to solve rank minimization problems and 
compare them with the methods studied in \cite{CaCaSh10,JaMeDh10}.  This is left as a future research. 

\section*{Acknowledgment}

The author would like to thank Ting Kei Pong for proofreading and suggestions which substantially 
improves the presentation of the paper.     

\gap

\gap

\gap


\begin{thebibliography}{10}

\bibitem{BaBo88}
J. Barzilai and J.M. Borwein. 
\newblock Two point step size gradient methods. 
\newblock {\em IMA J. Numer. Anal.}, 8:141--148, 1988.

\bibitem{BlDa08}
T.~Blumensath and M.~E.~Davies. 
\newblock Iterative thresholding for sparse approximations. 
\newblock {\em J. FOURIER ANAL. APPL.}, 14:629--654, 2008.

\bibitem{BlDa09}
T.~Blumensath and M.~E.~Davies.  
\newblock Iterative hard thresholding for compressed sensing. 
\newblock {\em Appl. Comput. Harmon. Anal.}, 27(3):265--274, 2009.

\bibitem{BiMaRa00}
E.~G.~Birgin, J.~M.~Mart\'inez, and M.~Raydan.
\newblock Nonmonotone spectral projected gradient methods on convex sets.
\newblock {\em SIAM J. Optim.}, 4:1196--1211, 2000.

\bibitem{CaCaSh10}
J. Cai, E.~Cand\`{e}s, and Z. Shen. 
\newblock A singular value thresholding algorithm for matrix completion.
\newblock {\em SIAM J. Optim.}, 20:1956--1982, 2010.

\bibitem{HeGiTr06}
K.~K.~Herrity, A.~C.~Gilbert, and J.~A.~Tropp. 
\newblock Sparse approximation via iterative thresholding. 
\newblock {\em IEEE International Conference on Acoustics, Speech and Signal Processing}, 2006.

\bibitem{HuLiLu12}
J. Huang, S. Liu, and Z. Lu.
\newblock Sparse approximation via nonconvex regularizers.
\newblock Working paper, Department of Statistics, Texas A\&M University, 2012.

\bibitem{JaMeDh10}
P. Jain, R. Meka, and I. Dhillon.
\newblock Guaranteed rank minimization via singular value projection.
\newblock {\em Neural Information Processing Systems}, 937--945, 2010.

\bibitem{LaMo12}
G. Lan and R. D. C. Monteiro.
\newblock Iteration-complexity of first-order penalty methods for convex programming.
\newblock To appear in {\em Math. Prog.}.

\bibitem{LuZh12}
Z. Lu and Y. Zhang.
\newblock Sparse approximation via penalty decomposition methods.
\newblock Manuscript, Department of Mathematics, Simon Fraser University, Februray 2012.

\bibitem{MaZh93}
S.~Mallat and Z.~Zhang.
\newblock Matching pursuits with time-frequency dictionaries.
\newblock {\em IEEE T. Image Process.}, 41(12):3397--3415, 1993.

\bibitem{Nes04}
Y. Nesterov.
\newblock {\em Introductory Lectures on Convex Programming: a basic course.} 
\newblock Kluwer Academic Publishers, Massachusetts, 2004.

\bibitem{Nik11}
M. Nikolova.
\newblock Description of the minimizers of least squares regularized with $l_0$ norm.
\newblock Report HAL-00723812, CMLA - CNRS ENS Cachan, France, October 2011.

\bibitem{Tr04}
J.~A.~Tropp. 
\newblock Greed is good: algorithmic results for sparse approximation.
\newblock {\em IEEE T. Inform. Theory}, 50(10):2231--2242, 2004.

\end{thebibliography}
\end{document}